\documentclass[11pt]{amsart}

\advance\hoffset by -0.5 truecm
\usepackage{amssymb}
\usepackage{amsmath}
\usepackage{graphicx}
\usepackage{enumitem}
\usepackage{color}
\newtheorem{Theorem}{Theorem}[section]

\newtheorem{Corollary}[Theorem]{Corollary}

\newtheorem{Proposition}[Theorem]{Proposition}

\newtheorem{Remark}[Theorem]{Remark}

\def \dim{{\mbox {dim}}\,}

\def\R\re
\def\V{\bf V}

\def \e{\varepsilon}

\def \re{{\mathbb R}}

\def \0{\lambda_{0}}

\def \N{{\mathbb N}}

\begin{document}
\title[]{The Equivariant Second Yamabe Constant}

\author[G. Henry]{Guillermo Henry}\thanks{G. Henry is partially supported by grants 20020120200244BA from  Universidad de Buenos Aires and  PIP 11220110100742 from CONICET
} 
 \address{Departamento de Matem\'atica, FCEyN, Universidad de Buenos
Aires, Ciudad Universitaria, Pab. I., C1428EHA,
          Buenos Aires, Argentina, and CONICET, Argentina.}
\email{ghenry@dm.uba.ar}

\author[F. Madani]{Farid Madani}\thanks{}
 \address{Institut f\"{u}r Mathematik,
Goethe Universit\"{a}t Frankfurt, Robert-Mayer-Str. 10
60325 Frankfurt am Main,  Germany}
\email{madani@math.uni-frankfurt.de}


\date{}

\begin{abstract} 
 For a  closed Riemannian manifold of dimension $n\geq 3$ and a subgroup $G$ of the isometry group,  we define and study the $G-$equivariant second Yamabe constant and we
obtain some results on the existence of $G-$invariant nodal solutions of the Yamabe equation.

\end{abstract}

\maketitle

\section{Introduction}

Let $(M^n,g)$ be a closed Riemannian manifold of dimension  $n\geq 3$. Let $[g]$ be the conformal class of $g$. Consider $G$  a compact subgroup of the isometry group $I(M,g)$. 
The $G-$invariant conformal class $[g]_G$ is the  subset of  $G$-invariant Riemannian metrics of $[g]$, {\em i.e.}, 
$$[g]_G:=\{fg:\ f\in C^{\infty}_{>0}(M),\ \sigma^*(f)=f\ \forall\ \sigma \in G\}.$$ 

The Yamabe functional $J:[g]\longrightarrow \re$ is  defined by:  
$$h\in [g]\longrightarrow J(h):=\frac{ \int_M s_h dv_h}{vol(M,h)^{\frac{n-2}{n}}},$$
where $s_h$ denotes the scalar curvature fo $h$.
The infimum of $J$ over $[g]$ is the conformal invariant $Y(M,[g])$ called  the Yamabe constant of $(M,[g])$. If
 we take the infimum of $J$ over $[g]_G$ instead of $[g]$, we obtain the $G-$equi\-va\-riant Yamabe constant:
$$Y_G(M,[g]_G):=\inf_{h\in [g]_G}J(h).$$ 
If $G$ is the trivial group, $Y_G(M,[g]_G)$ and $Y(M,[g])$ are equal. In case when  the constant $Y_G(M,[g]_G)$ is attained by a Riemannian metric $h$, then $h$ is $G-$invariant and has  constant scalar curvature. 

For a point $P\in M$, we denote by $O_G(P)$  the orbit of $P$ induced by the action of the group $G$. It is well known that $O_G(P)$ is a  compact submanifold of $M$. 
Throughout this article, we  use $\Lambda_{G}$ to denote the infimum of the cardinality of all the orbits induced by  the action of the group $G$:
$$\Lambda_G:=\inf_{P\in M}card \ O_G(P).$$ The $G-$equivariant Yamabe constant of  $(M,g)$ is bounded from above by 
$$Y_G(M,[g]_G)\leq Y(S^n)\Lambda_{G}^{\frac{2}{n}},$$
where $Y(S^n)$ is the Yamabe constant of $S^n$ endowed with the round metric $g^n_0$ (see \cite{Hebey-Vaugon2}).

Hebey and Vaugon proved in \cite{Hebey-Vaugon2} that the $G-$equivariant Yamabe constant is attained if the above inequality is strict.  
Moreover, they conjectured  that if $(M,g)$ is not conformal to the round sphere or if the action of $G$ has no fixed point, then 
\begin{equation}\label{H-Vconjecture}
 Y_G(M,[g]_G)<Y(S^n)\Lambda_G^{\frac{2}{n}}.
\end{equation}
This assertion is known in literature 
as the Hebey-Vaugon conjecture, and it is trivially satisfied when $\Lambda_G=\infty$ or $Y(M,[g])\leq 0$. 

Note that if the Yamabe constant is non-positive,  then there exists only one 
metric (of a given volume) of constant scalar curvature  in the conformal class $[g]$. This metric is invariant by the whole isometry group, and in particular it is invariant for any subgroup.  Therefore,  $Y_G(M,[g]_G)=Y(M,[g])$ for any $G\subseteq I(M,g)$, and Inequality \eqref{H-Vconjecture} holds.   

Assuming the Positive Mass Theorem (PMT) for higher dimensions,  He\-bey and Vaugon proved in \cite{Hebey-Vaugon2} the conjecture for some cases, {\em i.e.}, the action of the group $G$ is free; $3\leq \dim (M)\leq 11$;  there exists $P\in M$, with minimal finite orbit, such that $\omega(P)\geq (n-6)/2$ or $\omega(P)\in\{0,1,2\}$ (see the definition of $\omega(P)$ below).
 In \cite{Madani2}, the second author proved (without using the PMT) that  \eqref{H-Vconjecture}  holds, if there exists $P\in M$, with a minimal finite orbit, such that $\omega(P)\leq (n-6)/2$. Therefore, combining these results we conclude that the Hebey-Vaugon  conjecture is true, if we assume the PMT. 

 We say that $u$ is a solution of the Yamabe equation if for some constant $c$, the function $u$ satisfies 
\begin{equation}\label{Yamabeeq}
 L_g(u)=c|u|^{p_n-2}u,
\end{equation}
 where $L_g:=a_n\Delta_g+s_g$ is the conformal Laplacian of $(M,g)$, $a_n:=\frac{4(n-1)}{n-2}$, and $p_n:=\frac{2n}{n-2}$.
 A nodal solution of the Yamabe equation of $(M,g)$ is a solution of \eqref{Yamabeeq}
that changes sign. Recall that positive solutions of Yamabe equation are related with Riemannian metrics of constant scalar curvature in the conformal class of $[g]$. More precisely, if $u$ is a smooth positive solution of \eqref{Yamabeeq}, then $u^{p_n-2}g$ has constant scalar curvature $c$. By the resolution of the Yamabe problem, due to Yamabe \cite{Yamabe}, Tr\"{u}dinger \cite{Trudinger}, Aubin \cite{Aubin}, and Schoen \cite{Schoen1}, we know that there exists a positive solution of  \eqref{Yamabeeq} if and only if both  $Y(M,[g])$ and the constant $c$ have the same sign. However, it seems important to study the set of nodal solutions in order to understand the whole set of solutions of the Yamabe equation. 
Several authors have studied  nodal solutions of Yamabe type equations: Hebey and Vaugon \cite{Hebey-Vaugon4}, Holcman \cite{Holcman}, Jourdain \cite{Jourdain}, Djadli and Jourdain \cite{Djadli-Jourdain}, Ammann and Humbert \cite{Ammann-Humbert}, Petean \cite{Peteannodal}, and  El Sayed \cite{ElSayed},  just to mention some of them.

In \cite{Ammann-Humbert},  Ammann and Humbert introduced and studied the second Yamabe constant, which is  defined by 

$$Y^2(M,[g]):=\inf_{h\in [g]}\lambda_2(h)vol(M,h)^{\frac{2}{n}},$$

where   $\lambda_2(h)$ is the second eigenvalue of $L_h$. 

This constant is related with  nodal solutions of the Yamabe equation. They proved that for  a closed connected Riemannian manifold $(M,[g])$, with non-negative Yamabe constant, the second Yamabe constant in never realized by a Riemmannian metric. However, there is a chance to attain  $Y^2(M,[g])$, if we enlarge the conformal class $[g]$, by adding  symmetric tensors of the form 
$u^{p_n-2}g$, where $u\neq 0$ is a non-negative function which belongs to $L^{p_n}(M)$. These tensors are called {\it generalized metrics}. When $Y^2(M,[g])$ is positive and is attained by a generalized metric $u^{p_n-2}g$, Ammann and Humbert, in \cite{Ammann-Humbert}, proved that $u=|w|$, where $w$ is a nodal solution of the Yamabe equation on $(M,g)$.

In this article we are concerned by the  $G-$equivariant version of the second Yamabe constant. We assume that the action  of $G$ on $M$ is not transitive. Otherwise, the constant functions would be the only $G-$invariant functions, and therefore $s_g$ will be the only one eigenvalue of $L_g$ restricted to the set of $G-$invariant functions.  We denote by   
$\lambda_{2,G}(h)$ the second eigenvalue of the conformal Laplacian $L_h$ restricted to the Sobolev space $H_{1,G}^2(M)$ ({\em i.e.}, the set of $G-$invariant functions in $L^2(M)$, whose differential is also in $L^2(M)$). We define the $G-$equivariant second Yamabe 
constant as 

$$Y^2_G(M,[g]_G):=\inf_{h\in [g]_G}\lambda_{2,G}(h)vol(M,h)^{\frac{2}{n}}.$$
Note that when $G=\{Id_M\}$, $Y^2_G(M,[g]_G)=Y^2(M,[g])$.

Let $W$ be the Weyl tensor of $(M,g)$.  We define $\omega:M\longrightarrow \N_0 \cup \{+\infty\}$ as follows: 
If there exists $l_0\in \N_0$ such that $|\nabla^{l_0} W(P)|\neq 0$, then

$$\omega(P):= \inf\{l\in \N_0: \  |\nabla^l W(P)|\neq 0\}.$$

If $|\nabla^l W|(P)=0,\ \forall\ l\in \N_0$, then we set $$\omega(P):=+\infty.$$

For instance, if $(M,g)$ is locally conformally flat ($\dim(M)\geq 3$), then $\omega(P)=+\infty$ for any $P\in M$. If $M$ is not locally conformally flat, there exists $P\in M$ such that $\omega(P)=0$.

Our first result is the following theorem

\begin{Theorem}\label{strictinequalityA} 
Let $(M,g)$ be a closed Riemannian manifold of dimension $n\ge 3$ and $G$ be a closed subgroup of $I(M,g)$. Assume $Y_G(M,[g]_G)\geq 0$. 
The following strict inequality holds
\begin{equation}\label{E2ndYbounds}
 Y^2_G(M,[g]_G)< \Big(Y_G(M,[g]_G)^{\frac{n}{2}} +Y(S^n)^{\frac{n}{2}}\Lambda_G\Big)^{\frac{2}{n}}
\end{equation}
if one of the following items is satisfied:
\begin{enumerate}
\item[i)]  $\Lambda_G=\infty$. 
\item[ii)]  There exists a point $P\in M$ that belongs to a  minimal finite orbit, such that $\omega(P)\leq \frac{n-2}{6}$ and 
\begin{itemize}
\item $\dim (M)\geq 11$ if $Y_G(M,[g]_G)>0$.
\item  $\dim (M)\geq 9$ if $Y_G(M,[g]_G)=0$.
 \end{itemize}
\end{enumerate}
\end{Theorem}

\begin{Remark} Actually, when $G=\{Id_M\}$ and $(M,g)$ is not locally conformally flat, Theorem \ref{strictinequalityA} was proved by Ammann and Humbert \emph{(}Theorem 1.5  of \cite{Ammann-Humbert}\emph{).}          
\end{Remark}
 
If we do not assume any conditions on the orbits induced by the action of $G$, we have the following lower and upper bounds for the $G-$equivariant second Yamabe constant: 
 
\begin{Proposition}\label{boundsproposition} Let $(M,g)$ be a closed Riemannian manifold of dimension $n\ge 3$ and $G$ be a closed subgroup of $I(M,g)$. We assume that $Y_G(M,[g]_G)$ is attained and non-negative. Then, we have 
\begin{equation}\label{E2ndYboundslarge}
 2^{\frac{2}{n}}Y_G(M,[g]_G)\leq Y^2_G(M,[g]_G)\leq \Big(Y_G(M, [g]_G)^{\frac{n}{2}} +Y(S^n)^{\frac{n}{2}}\Lambda_G\Big)^{\frac{2}{n}}.
\end{equation}

\end{Proposition}

\begin{Remark} The assumption \lq\lq$Y_G(M,[g]_G)$ is attained\rq\rq in Proposition \ref{boundsproposition} can be removed, if we assume the Positive Mass Theorem for higher dimensions. Namely, under this assumption, the results proven in \cite{Hebey-Vaugon2} and \cite{Madani,  Madani2} imply that the Hebey-Vaugon conjecture is true and therefore, $Y_G(M,[g]_G)$ is attained. This is the case for instance, when $M$ admits a spin structure.  
\end{Remark}

With  $ L^{p_n}_{G,\geq 0}(M)$ we denote the non-negative $G-$invariant functions that belong to $L^{p_n}(M)-\{0\}$. We define the $G-$invariant generalized conformal class $[g]_{G,gen}$ by 

$$[g]_{G,gen}:=\{u^{p_n-2}g:\ u\in L^{p_n}_{G,\geq 0}(M)  \},$$
and any element of $[g]_{G,gen}$ is called a {\em $G$-generalized metric}.

Using the variational characterization of the eigenvalues of the conformal Laplacian, we  enlarge the definition of $\lambda_{2,G}(h)$  to the set of $G$-generalized metrics. We do the the same extension for $vol(M,h)$.  Therefore, as well as for the second Yamabe constant (see \cite{Ammann-Humbert}), we have 
\begin{equation}\label{2ndYamabegen}Y^2_G(M,[g]_G)=\inf_{h\in [g]_{G,gen} }\lambda_{2,G}(h)vol(M,h)^{\frac{2}{n}}.  
\end{equation}

The relationship between the $G-$equivariant second Yamabe constant and the nodal solutions of the Yamabe equation is given by the following result:

\begin{Theorem}\label{nodalsolution} Let $(M,g)$ be a closed Riemannian manifold of dimension $n\ge 3$ and $G$ be a closed subgroup of $I(M,g)$. Assume that $Y_G(M,[g]_G)>0$ and $Y_G^2(M,[g]_G)$ is attained by a $G-$generalized metric $h=u^{p_n-2}g$. Then,  
$u=|w|$, where $w\in C^{3,\alpha}(M)\cap C^{\infty}(M-\{w=0\})$, with   $0<\alpha\leq 4/(n-2)$, is a $G-$invariant nodal solution of the Yamabe equation.  In particular, if  $\|u\|_{p_n}=1$, then 
$$L_g(w)=Y_G^2(M,[g]_G)|w|^{p_n-2}w.$$  
\end{Theorem}

\begin{Remark} Theorem \ref{nodalsolution} implies that the $G-$equivariant second Yamabe constant of a connected closed Riemannian manifold  with $Y_G^2(M,[g]_G)>0$  is never a\-ttai\-ned by a Riemannian metric. Indeed, if we assume that $Y_G^2(M,[g]_G)$  is attained by $h$, then by Theorem \ref{nodalsolution}, we have that  $h=|w|^{p_n-2}g$, with $w$ a  changing sign function. Since $M$ is connected, $w$ must be zero somewhere, and therefore $h$ is not a Riemannian metric. 
\end{Remark}

Now, we give a sufficient condition, in order that the $G-$equivariant se\-cond Yamabe constant to be attained by a generalized metric 

\begin{Theorem}\label{Y2attained} Let $(M,g)$ be a closed Riemannian manifold of dimension $n\geq 3$ and $G\subset I(M,g)$ be a compact subgroup. Assume that $Y_G(M,[g]_G)$ is non-negative. The $G-$equivariant second Yamabe constant is attained  by a generalized metric if 
$$Y^2_G(M,[g]_G)< \Big(Y_G(M,[g]_G)^{\frac{n}{2}} +Y(S^n)^{\frac{n}{2}}\Lambda_G\Big)^{\frac{2}{n}}.$$
\end{Theorem}

As a consequence of Theorem \ref{Y2attained} and Proposition \ref{nodalY2=0} (see Section \ref{sectionnodalsolutions}), we have

\begin{Corollary}\label{CorolNodalsolution}
Let $(M,g)$ be a closed Riemannian manifold that satisfy the assumptions of Theorem \ref{strictinequalityA}. Then, there exists a $G-$invariant nodal solution of the Yamabe equation on $(M,g)$.
\end{Corollary}

\subsection{Riemannian products}

\ 

\medskip

For a non-compact Riemannian manifold $(W^n,h)$, we define the $G-$equi\-variant Yamabe constant by taking the infimum of the Yamabe functional over the space of $G-$invariant smooth functions compactly supported $C^{\infty}_{0,G}(W)$, {\em i.e.},  
$$Y_G(W,[h]_G):=\inf_{u\in C^{\infty}_{0,G}(W)-\{0\}}\frac{\int_W a_{n}|\nabla u|_h^2+s_hu^2dv_h}{\|u\|_{p_n}^2}.$$

Note that for closed Riemannian manifolds $(M^n,g)$, $(N^l,h)$ and for $G_1$, $G_2$ two closed subgroups of $I(M,g)$ and $I(N,h)$, respectively, the group   $G:=G_1\times G_2$ is a closed subgroup of 
$I(M\times N,g+th)$ for any $t\in \re_{>0}$.

\begin{Theorem}\label{asyntotic} Let $(M^n,g)$ be a closed Riemannian manifold of dimension $n\geq 2$ with  positive scalar curvature,  and $(N^l,h)$ be any closed Riemannian manifold of dimension $l$.  
Let  $G$ be a closed subgroup of $I(M,g)$, where its action is trivially extended to $M\times N$ and $M\times \re^l$.
 
Then,
\begin{gather}\lim_{t\to +\infty}Y_G(M\times N,[g+th]_G)=Y_G(M\times \re^l,[g+g^l_e]_G),\label{itema}\\
\lim_{t\to +\infty}Y^2_G(M\times N,[g+th]_G)=2^{\frac{2}{n+l}}Y_G(M\times \re^l,[g+g^l_e]_G), \label{itemb}
\end{gather}
where  $(\re^l,g^l_e)$ is the $n-$dimensional Euclidean space.
\end{Theorem}

\begin{Theorem}\label{TC} Let $(M^n,g)$ be a closed Riemannian manifold of dimension $n\geq 2$ with positive scalar curvature.  If $G=G_1\times Id_{\re^l}$,  with $G_1$  a compact subgroup of $I(M,g)$ and $l\geq 2$, then 
\begin{equation}\label{TCineq}Y_G(M^n\times \re^l,[g+g_e^l]_G)<Y(S^{n+l})\Lambda_{G}^{\frac{2}{n+l}}.\end{equation}
\end{Theorem}
One mentions here that using Equality  \eqref{itema} and Inequality \eqref{TCineq}, we obtain 
$$ Y_G(M\times N,[g+th]_G)< Y(S^{n+l})\Lambda_{G}^{\frac{2}{n+l}},$$
for $t$ large enough.
This means that the Hebey-Vaugon conjecture holds for $(M\times N,[g+th]_G)$ and $Y_G(M\times N,[g+th]_G)$ is attained.

The following corollary is an immediate consequence of the theorem above and Theorem \ref{asyntotic}: 

\begin{Corollary}\label{strictineqforproduct} Assume the same assumptions as in Theorem \ref{asyntotic}. If in addition $\dim(N)=l\geq 2$, then for $t$ large enough, we have 
$$Y^2_G(M^n\times N^l,[g+th]_G)<\Big(Y_G(M,[g]_G)^{\frac{n+l}{2}} +Y(S^{n+l})^{\frac{n+l}{2}}\Lambda_G\Big)^{\frac{2}{n+l}}.$$
 \end{Corollary}

Hence, assuming the same hypotheses as Corollary \ref{strictineqforproduct}, by Theorem \ref{nodalsolution} and Theorem \ref{Y2attained}, we obtain 

 \begin{Corollary}\label{nodalsolproduct} For $t$ large enough there exists a $G-$invariant nodal solution of the Yamabe equation on $(M\times N,g+th)$.  
\end{Corollary}

\subsection{The subcritical case}

$\ $

\medskip

Let $(M^n,g)$ and $(N^l,h)$ be closed Riemannian manifolds with $n+l\geq 3$. Assume that $(N,h)$ has constant scalar curvature. Let $G=G_1\times G_2$ be a compact subgroup of $I(M\times N, g+h)$ where   $G_1$ and $G_2$ are closed subgroups of $I(M,g)$ and $I(N,h)$, respectively. If the action of $G_1$ on $M$ is not transitive, we define the $G_1-$equivariant $M-$second Yamabe constant as
$$Y^2_{M,G_1}(M\times N,g+h):=\inf_{\bar{g}\in [g+h]_{M,G_1}}\lambda_{2,G_1}^M(\bar{g})vol(M\times N,\bar{g})^{\frac{2}{n+l}}.$$
where $[g+h]_{M,G_1}$ is the $(M,G_1)-$
 conformal class of $(g+h)$ defined by
$$[g+h]_{M,G_1}:=\{u^{p_{n+l}-2}(g+h):u \in C^{\infty}_{G_1,>0}(M)\}$$  and  $\lambda_{2,G_1}^M(\bar{g})$ is the second eigenvalue of  the operator 
$L_{\bar{g}}$ restricted to  functions of 
$H^2_{1,G}(M\times N)$ that are constant on $N$. By definition, this eigenvalue does not depends on the group $G_2$. 

As above, we can define the generalized $(M,G_1)-$
 conformal class of $(g+h)$, named $[g+h]_{M,G_1,gen}$, by adding to  $[g+h]_{M,G_1}$ the symmetric tensors of the form $u^{p_{n+l}-2}(g+h)$ with $u\in  L^{p_{n+l}-2}_{G_1,\geq 0} (M)$. We call a generalized $(M,G_1)-$metric to any tensor that belongs to $[g+h]_{M,G_1,gen}$. In this case we have  a similar equality as in \eqref{2ndYamabegen}.

\begin{Proposition}\label{subcritical} Let $(M^n,g)$ and $(N^l,h)$ be closed Riemannian manifolds with $n+l\geq 3$. Assume that $g$ has constant scalar curvature and $s_g+s_h$ is positive. Let $G=G_1\times G_2$ be a closed subgroup of $I(M\times N, g+h)$ where   $G_1$ and $G_2$ are closed subgroups of $I(M,g)$ and $I(N,h)$, respectively. Then, the $G_1-$equi\-va\-riant $M-$second Yamabe constant is always achieved by a generalized $(M,G_1)-$metric. 
\end{Proposition}

As a consequence, we obtain the following result.

\begin{Corollary}\label{subcriticalnodal} Let $(M^n,g)$, $(N^l,h)$ and $G_1\subseteq I(M,g)$ as in Proposition \ref{subcritical}. There exists a $G_1\times\{Id_N\}$-invariant  nodal solution of the Yamabe equation on $(M\times N, g+h)$ that depends only on the $M-$variable.  
\end{Corollary}

\section{Preliminaries}

Let $(M,g)$ be a closed Riemannian manifold of dimension $n\geq 3$. If $h=u^{p_{n}-2}g$, then 

$$J(h)=\frac{\int_M a_n|\nabla u|^2_g+s_gu^2dv_g}{\|u\|^2_{p_{n}} }.$$

We define $J_g:C^{\infty}_{>0}(M)\longrightarrow \re$ by $$J_g(u)=\frac{\int_M a_n|\nabla u|^2_g+s_gu^2dv_g}{\|u\|^2_{p_{n}} }.$$
We can extend the functional  $J_g$ to $H^2_1(M)-\{0\}$ and we obtain that
$$Y(M,[g])=\inf_{u\in H^2_1(M)-\{0\}}J_g(u)$$
From now on, we skip the subscript of the background metric $g$ in $J_g$, and we call this functional as the Yamabe functional as well. 

Let $G$ be a compact subgroup of $I(M,g)$. The $G$-equivariant Yamabe constant of $(M,g)$ is 

$$Y_G(M,[g]_G)=\inf_{u\in H^2_{1,G}(M)-\{0\}}\frac{\int_M a_n|\nabla u|^2_g+s_gu^2dv_g}{\|u\|^2_{p_{n}} }$$

When the $G-$equivariant Yamabe constant is attained, it is realized  by a positive smooth $G-$invariant function $u$, that induces a $G-$invariant metric $g_u:=u^{p_n-2}g$ of constant scalar curvature.

Let $h$ be a $G-$invariant metric in the conformal class $[g]$. We assume that the action of the group $G$ is not transitive, hence the spectrum of $L_h$ res\-tricted to $H^2_{1,G}(M)$ is a sequence of non-decreasing eigenvalues 
$$\lambda_{1,G}(h)<\lambda_{2,G}(h)\leq\lambda_{3,G}(h)\leq\dots\leq \lambda_{k,G}(h)\leq \dots$$

We obtain the following variational characterizations:

\begin{equation}\label{firsteigenvalue}\lambda_{1,G}(h)=\inf_{v\in H^2_{1,G}(M)-\{0\}}\frac{\int_M a_n|\nabla v|^2_h+s_hv^2dv_h}{\|v\|^2_{2} },\end{equation}
and 
\begin{equation}\label{secondeigenvalue}\lambda_{2,G}(h)=\inf_{V\in Gr^2(H^2_{1,G}(M))}sup_{v\in V-\{0\}}\frac{\int_M a_n|\nabla v|^2_h+s_hv^2dv_h}{\|v\|^2_{2} },
 \end{equation}
where $Gr^k(H^2_{1,G}(M))$ denotes the space of $k-$dimensional subspaces of $H^2_{1,G}(M)$.
The signs of  $Y_G(M,[g]_G)$ and $\lambda_{1,G}(g)$ coincide. If $Y_G(M,[g]_G)\geq 0$,  using \eqref{firsteigenvalue}, \eqref{volume}, and the H\"{o}lder inequality, it can be proven that
 $$Y_G(M,[g]_G)=\inf_{h\in[g]_G}\lambda_{1,G}(h)vol(M,h)^{\frac{2}{n}}.$$

 Let $u\in L^{p_n}_{G,\geq 0}(M)-\{0\}$. With  $Gr^k_u(H^2_{1,G}(M))$, we denote the space of $k-$dimensional subspace of $H^2_{1,G}(M)$ which are also $k$-dimensional  subspaces in $H^2_{1,G}(M)-\{u=0\}$, respectively. 
 
We have the following variational characterization  for the $G-$equivariant second Yamabe cons\-tant and  the $G-$equivariant $M-$second Yamabe cons\-tant:

\begin{Proposition}\label{variationalcharact} 
Let $(M^n,g)$ and $(N^l,h)$ be closed Riemannian manifolds. Let $G=G_1\times G_2$ be a closed subgroup of $I(M\times N, g+h)$ where   $G_1$ and $G_2$ are closed subgroups of $I(M,g)$ and $I(N,h)$, respectively. Then

\begin{multline*}Y^2_G(M,[g]_G)=\inf_{\substack{
		u\in C^{\infty}_{G,>0}(M)\\
		V\in Gr^2(H^2_{1,G}(M))}}\sup_{v\in V-\{0\}}\frac{\int_M a_n|\nabla v|^2_g+s_gv^2dv_g}{\int_Mu^{p_{n}-2}v^2dv_g}(\int_Mu^{p_n}dv_g)^{\frac{2}{n}},\\
Y^2_{M,G}(M^n\times N^l,g+h)=\inf_{\substack{
		u\in C^{\infty}_{G,>0}(M)\\
		V\in Gr^2(H^2_{1,G_1}(M))}}\sup_{v\in V-\{0\}}\frac{\int_M a_{n+l}|\nabla v|^2_g+s_gv^2dv_g}{\int_Mu^{p_{n+l}-2}v^2dv_g}\\
\times\big(\int_Mu^{p_{n+l}}dv_g\big)^{\frac{2}{n+l}} vol(N,h)^{\frac{2}{n+l}}.
\end{multline*}
\end{Proposition}

\begin{proof} If $g'\in [g]_G$,  we write $g'=u^{p_n-2}g$ for some $u\in C^{\infty}_{G,>0}(M)$.  The volume of $(M,g')$ is 
\begin{equation}\label{volume}
 vol(M,g')=\int_M u^{p_n}dv_g.
\end{equation}

Using  the conformal invariance property of $L_{g'}$ ({\em i.e.}, $L_{g'}(v)=u^{1-p_n}L_g(uv)$) we obtain for any $v\in C^{\infty}(M)-\{0\}$ that 
 $$\frac{\int_M L_{g'}(v)vdv_{g'}}{\int_M v^2dv_{g'}}=\frac{\int_M L_g(uv)uvdv_g }{\int_M (uv)^2u^{p_n-2}dv_g}.$$

Therefore, from  \eqref{secondeigenvalue} it follows that 

\begin{equation}\label{generlized}
\lambda_{2,G}(g')vol(M,g')^{\frac{2}{n}}=
\end{equation}
$$=\inf_{\substack{
            V\in Gr^2(H^2_{1,G}(M))}}\sup_{v\in V-\{0\}}\frac{\int_M a_n|\nabla v|^2_g+s_gv^2dv_g}{\int_Mu^{p_{n}-2}v^2dv_g}(\int_Mu^{p_n}dv_g)^{\frac{2}{n}}.
            $$

Then, taking the infimum over $C^{\infty}_{>0,G}(M)$ we get the proposition.
In a similar way we prove the variational characterization for the $G-$equiva\-riant $M-$second Yamabe constant of the product manifold $M\times N$.
\end{proof}

If $Y_G(M,[g]_G)\geq 0$ we can extend naturally the definition of $\lambda_{2,G}$ and $vol(M)$ to $G-$invariant generalized metrics conformal to $g$ (when $G$ is the trivial group see for instance \cite{Ammann-Humbert}). More precisely, if $g'=u^{p_n-2}g$, with $u\in L^{p_n}_{G,\geq 0}-\{0\}$, we define  $\lambda_{2,G}(g')$ and $vol(M,g')$ as
\begin{equation}\label{secondeigenvalueGen}\lambda_{2,G}(g'):=\inf_{\substack{
            V\in Gr_u^2(H^2_{1,G}(M))}}\sup_{v\in V-\{0\}}\frac{\int_M a_n|\nabla v|^2_g+s_gv^2dv_g}{\int_Mu^{p_{n}-2}v^2dv_g},\end{equation}
and $$vol(M,g')=\int_Mu^{p_n}dv_g.$$

By Proposition \ref{variationalcharact} we have that 
$$Y^2_G(M,[g]_G)=\inf_{\substack{
            u\in L^{p_n}_{G,\geq 0}(M)\\
            V\in Gr^2_u(H^2_{1,G}(M))}}\sup_{v\in V-\{0\}}\frac{\int_M a_n|\nabla v|^2_g+s_gv^2dv_g}{\int_Mu^{p_{n}-2}v^2dv_g}(\int_Mu^{p_n}dv_g)^{\frac{2}{n}}.$$
 Then, using the definitions of  $\lambda_{2,G}(g')$ and $vol(M,g')$   we get that 
 $$Y^2_G(M,[g]_G)=\inf_{g'\in [g]_{G,gen} }\lambda_{2,G}(g')vol(M,g')^{\frac{2}{n}}.$$

\section{Bounds for the second equivariant Yamabe constant}

The proof of Theorem \ref{strictinequalityA} and Proposition \ref{boundsproposition}  require some test functions that we will introduce in the next subsections.

\subsection{The test functions}

\subsubsection{The case $\omega=0$.}

$\ $
\medskip

In this subsection we recall the classical Aubin's test functions and their equivariant version. These are the test functions that we  use to prove Proposition \ref{boundsproposition}. They also work to prove Theorem \ref{strictinequalityA}, when the  Weyl tensor does not vanish at a  minimal finite orbit.

For $P\in M$ and $\delta>0$ small, let us  consider  the following function:
\begin{equation}\label{testfunction1}
 \phi_{P,\varepsilon}(Q)=C_{\varepsilon}\eta(Q)\Big(\frac{\varepsilon}{\varepsilon^2+r^2(Q)} \Big)^{\frac{n-2}{2}},
\end{equation}
where $\eta$ is a non-negative radial cut-off function centered at $P$, such that $|\eta|\leq 1$,  $\eta(Q)=1$ for $Q\in B_{\delta}(P)$,  $\eta(Q)=0$ for $Q\in M-B_{2\delta}(P)$, and $|\nabla \eta |\leq  2/\delta$; $C_{\varepsilon}$ is the unique positive constant such that $\|\phi_{P,\varepsilon}\|_{p_n} =1$; here  $r$ denote the distance to  $P$.
It is well known that  
\begin{equation}\label{limittestfucntion} J(\phi_{P,\varepsilon})
\longrightarrow_{\varepsilon \to 0} Y(S^{n}).
\end{equation}

Assume the orbit of $P$ induced by the action of $G$ is finite and its cardinality is $k$. Let us write it as $O_G(P)=\{P_1, P_2,\dots,P_k\}$, where $P_1=P$. For $\delta$ small enough, we set
\begin{equation}\label{equivtestfunction1}
 \phi_{\varepsilon}(Q):=C\sum_{i=1}^{k}\phi_{P_i,\varepsilon}(Q),
\end{equation}
where $C$ is a chosen constant such that  $\|\phi_{\varepsilon}\|_{p_n}=1 $.  The function $\phi_{\varepsilon}$ is $G-$invariant, and  by \eqref{limittestfucntion} we have that 

\begin{equation}\label{asymptotically}
J(\phi_{\epsilon})\longrightarrow_{\varepsilon \to 0} Y(S^{n})k^{\frac{2}{n}}.
 \end{equation}
 
Standard computations shows that 
\begin{equation}
cf(\varepsilon)\leq \int_{M} \phi_{P,\epsilon}^{q}dv_g \leq Cf(\varepsilon),
\end{equation}
where \begin{equation}\label{estimationpowerstest}f(\varepsilon):=
\begin{cases}
 \varepsilon^{\frac{2n-(n-2)q}{2}} & \mbox{if }\  q>\frac{n}{n-2},\\  \ln(\delta\varepsilon^{-1})\varepsilon^{\frac{n}{2}} & \mbox{if }\    q=\frac{n}{n-2},\\ \varepsilon^{\frac{(n-2)q}{2}} & \mbox{if }\   q<\frac{n}{n-2} .\end{cases}
 \end{equation}

 If $\omega(P)=0$,  we use the classical estimations for the Aubin's test functions  when $(M,g)$  non-locally conformally flat (see \cite{Aubin}) to conclude that
 
 \begin{equation}\label{estimationAubin}
 J(\phi_{\varepsilon}):=\begin{cases} Y(S^n)\Lambda_G^{\frac{2}{n}}-A(M)\varepsilon^4+o(\varepsilon^4) & \mbox{if }\  n>6,\\  Y(S^n)\Lambda_G^{\frac{2}{n}}-A(M)\varepsilon^4|ln(\varepsilon^2)|+o(\varepsilon^4|ln(\varepsilon^2)|) &  \mbox{if } \  n=6 ,\end{cases} 
 \end{equation}
where $A(M)$ is a positive constant.
Therefore,  for $\omega(P)=0$ and $\varepsilon$ small enough,  we obtain that $$J(\phi_{\varepsilon})<Y(S^n)\Lambda_G^{\frac{2}{n}}.$$

 The test functions $\phi_{\varepsilon}$ do not work to prove Theorem \ref{strictinequalityA} when the Weyl tensor of $(M,g)$ vanishes at any minimal finite orbit. Within the cases covered by the Theorem, this means that $1\leq\omega (P)\leq\frac{n-6}{2}$ for any $P$ in a minimal finite orbit.  In order to prove Theorem \ref{strictinequalityA},  we need to consider other kind of test functions that we introduce  in the next subsection.

\subsubsection{The case $\omega\leq \frac{n-6}{2}$.}\label{omegacond}

 $\ $
 
 \medskip
 
We now start by some preliminaries, in order to introduce the test function.   
We assume that there exists a point $P\in M$ with finite minimal orbit, such that $\omega(P)\leq\frac{n-6}{2}$. For simplicity, we drop the letter $P$ on $\omega(P)$. Consider the geodesic normal coordinate system $\{x^j\}$ around $P$, such that $\det g=1+O(|x|^N)$, for some $N\in\N$ sufficiently large. We define the polar coordinates as $r:=|x|$, $\xi:=\frac{x}{r}$. 
Hebey and Vaugon, proved that there exists a $G$-invariant  metric $g'$ conformal to $g$, such that $|\nabla^jRiem_{g'}(P)|=0$, for any $j<\omega$ and $\Delta^{j}s_{g'}(P)=0$, $\nabla\Delta^{j}s_{g'}(P)=0$, for any $j<\omega+1$ (see \cite[Lemma~12 and  Lemma~8~bis]{Hebey-Vaugon2}). 
 The last two equalities imply that $$r^{1-n}\int_{\partial B_r(P)}s_{g'}dv_{g'}=O(r^{2\omega+2}).$$
 
From now on, without loss of generality, we assume the $|\nabla^jRiem_{g}(P)|=0$, for any $j<\omega$. 

We denote by $\mu$ the degree of the leading part in the Taylor expansion of the scalar curvature $s_g$. Namely,
$$s_g=r^{\mu} \sum_{|\alpha|=\mu}\nabla_\alpha s_g(P)\xi^\alpha+O(r^{\mu+1}).$$
 Set $\bar s:=r^{\mu} \sum_{|\alpha|=\mu}\nabla_\alpha s_g(P)\xi^\alpha$.  Since all derivatives of $s_g$ at $P$ vanish up to the order $\omega-1$, it follows that $\mu\geq \omega$.   Moreover, if $\mu>\omega$, Aubin, in \cite{Aubin2}, proved that $\int_{\partial B_r(P)}s_gdv<0$ (see also \cite{Madani3}). The last inequality is sufficient to have the estimate \eqref{estimationAubin}, using the same test function $\phi_\e$, introduced above (see for instance in \cite{Madani}). 

Now, we consider the case $\mu=\omega$. Thus, $\bar s$ is a homogeneous polynomial of degree $\omega$, with zero average over the unit sphere, since $r^{1-n}\int_{\partial B_r(P)}s_gdv_g=O(r^{2\omega+2})$. 

We claim that $\Delta^{[\frac{\omega}{2}]}\bar s\equiv 0$, where $[\frac{\omega}{2}]$ is the integer part of $\frac{\omega}{2}$. Indeed, if $\omega$ is even, then $\Delta^{[\frac{\omega}{2}]}\bar s$ is a constant, which vanishes at $P$. If $\omega$ is odd, then $\Delta^{[\frac{\omega}{2}]}\bar s$ is linear form, with  $\nabla\Delta^{[\frac{\omega}{2}]}\bar s(P)=0$.  

On the other hand, for any homogeneous polynomial  $f$ of degree $k$,   $r^{2}\Delta f=\Delta_S f-k(n+k-2)f$, where $\Delta_S$ is the spherical Laplacian. Applying this identity for $f:=\Delta^{[\frac{\omega}{2}]-1}\bar s$, we obtain
$$0=\Delta^{[\frac{\omega}{2}]}\bar s=r^{-2}(\Delta_S-\nu_{[\frac{\omega}{2}]}{\rm id})\Delta^{[\frac{\omega}{2}]-1}\bar s=
r^{-2[\frac{\omega}{2}]}\prod_{k=1}^{[\frac{\omega}{2}]}(\Delta_S-\nu_{k}{\rm id})\bar s,$$
where $\nu_k:=(\omega-2k+2)(n+\omega-2k)$. It follows that
$$\bar s|_{S^{n-1}}\in \bigoplus_{k=1}^{[\frac{\omega}{2}]} E_k,$$
where $E_k$ is the eigenspace associated to the positive eigenvalues  $\nu_k$ of the Laplacian $\Delta_S$ on the sphere $S^{n-1}$. Hence, there exists $\varphi_k\in E_k$, such that $$\bar s=r^{\omega}\sum_{k=1}^{[\frac{\omega}{2}]}\nu_k\varphi_k=r^{\omega}\Delta_S\sum_{k=1}^{[\frac{\omega}{2}]}\varphi_k.$$

Now, we introduce the test function that we use to prove the estimate \eqref{estimationAubin}:
\begin{equation}
\psi_{P,\e}(Q)=(1+r^{\omega+2}\sum_{k=1}^{[\frac{\omega}{2}]}c_k\varphi_k(\xi))\phi_{P,\e}(Q), 
\end{equation}
where $c_k$ are some real constants.  
In the computation of Yamabe functional of $\psi_{P,\e}$, the integrals
$$\int_M|\nabla \psi_{P,\e}|^2dv_g\quad {\rm and }\quad \int_M\psi_{P,\e}^{p_n}dv_g$$
are computed, by a straightforward computation, in terms of the $L^2$-norm of the $\varphi_k$'s. However, the integral $$\int_{M}s_g\psi_{P,\e}^2dv_g$$ 
 has to be computed carefully. Indeed, using a tricky decomposition of the metric $g$, we estimate the following integral in terms of the $L^2(S^{n-1})$-norm of the $\varphi_k$'s:
\begin{gather*}
r^{1-n}\int_{\partial B_r(P)}s_gdv_g\leq r^{2\omega+2}\sum_{k=1}^{[\frac{\omega}{2}]}e_k\|\varphi_k\|_{L^2}^2+o(r^{2\omega+2}),
\end{gather*}
where $e_k:=􏰎	\left(\frac{n-3}{4(n-2)}- \frac{(n-1)^2+(n-1)(\omega+2)^2}{4(n-2)(\nu_k-n+1)}\right)\nu_k$ (see in \cite[Lemma 2.3]{Madani}). It yields that for $2\omega<n-6$
\begin{equation*}
\begin{split}
\int_{M}s_g\psi_{P,\e}^2dv_g &=\int_{B_\delta(P)}\hspace{-0.5cm}s_g\phi_{P,\e}^2dv_g+2\int_0^\delta\phi_{P,\e}^2 r^{\omega+2}\sum_{k=1}^{[\frac{\omega}{2}]} c_k\int_{\partial B_r(P)}\hspace{-0.5cm}\varphi_k\bar s d\sigma dr+O(\e^{n-2})\\
& \leq \e^{2\omega+4}c(n,\omega)
 \sum_{k=1}^{[\frac{\omega}{2}]}(e_k+2c_k\nu_k)\|\varphi_k\|^2_{L^2}+O(\e^{n-2}),
 \end{split}
\end{equation*}
where $c(n,\omega)$ is a positive constants, which depends on $n$ and $\omega$. Note that for $2\omega=n-6$, we have a similar estimate, as above, of order $\e^{2\omega+4}\ln \e^{-1}$.  
For $2\omega<n-6$, we obtain that 

\begin{multline}\label{ineq w<n-6/2}
J(\psi_{P,\e})\leq Y(S^n)+c(n,\omega)
 \sum_{k=1}^{[\frac{\omega}{2}]}(e_k(n-2)^2+d_kc_k^2+2(n-2)^2c_k\nu_k)\|\varphi_k\|^2_{L^2}\e^{2\omega+4}
 \\+O(\e^{n-2}),
\end{multline}  
where $d_k:=4\left[(n-1)(n-2)\nu_k-n(n-2)^2+(\omega+2)^2(n^2+n+2)\right]$.  Inequality \eqref{ineq w<n-6/2} holds, when $2\omega=n-6$, 
with $\e^{2\omega+4}\ln \e^{-1}$ instead of $\e^{2\omega+4}$. In  \cite[Lemme 3.1]{Madani}, it was proven that for $c_k:=-\frac{(n-2)^2}{d_k}\nu_k$, the real number
$$e_k(n-2)^2+d_kc_k^2+2(n-2)^2c_k\nu_k= (n-2)^2\left(e_k-\frac{(n-2)^2}{d_k}\nu_k^2\right) $$
is negative, for any $k\leq [\frac{\omega}{2}]$. 
Hence, for $\e$ sufficiently small $J(\psi_{P,\e})< Y(S^n)$.

Finally, let $H\subset G$ be the stabilizer of $P$. 
We claim that the function
\begin{equation}\label{testfunctionw}
\psi_\e:=\sum_{\sigma\in G/H}\psi_{P,\e}\circ\sigma^{-1}
\end{equation}
is $G$-invariant. Indeed $\bar s$ is $H$-invariant, since the scalar curvature is invariant under the action of the isometry group. On the other hands, for any $\sigma\in H$, $\sigma_*(P)\colon (T_PM,g_P)\rightarrow (T_PM,g_P)$ is a linear isometry, where $g_P$ is the Euclidean metric. Thus, $\sigma_*(P)$ preserves the spheres, the Laplacian and its eigenfunctions $\varphi_k$.    Therefore,  $\psi_{P,\e}$ is $H$-invariant and $\psi_\e$ is $G$-invariant. We conclude that there exists some positive constant $A(n,\omega,G)$ such that  \begin{equation}\label{HebeyVaugonconjw}
J(\psi_{\e})\leq\begin{cases}
Y(S^n)\Lambda_G^{\frac{2}{n}}-A(n,\omega,G)\e^{2\omega+4}, & {\rm if }\; 2\omega<n-6;\\
Y(S^n)\Lambda_G^{\frac{2}{n}}-A(n,\omega,G)\e^{n-2}\ln(\e^{-1}), & {\rm if }\; 2\omega=n-6, 
\end{cases}
\end{equation} 
which implies that  \eqref{estimationAubin} holds for $\psi_\e$.

 \subsection{Proof of Theorem \ref{strictinequalityA}}
 
$\ $

\medskip

By assumption, we assume that either the orbits of $G$ are not finite or there exists a  finite minimal orbit $O_G(P)$ such that $\omega(P)\leq \frac{n-6}{2}$. If  
the orbits are not finite, Hebey and Vaugon proved in \cite{Hebey-Vaugon3} that the inclusion map of $H^2_{1,G}(M)$ in $L^{p_n}_G(M)$ is a compact operator. Therefore, by the classical variational method,  $Y_G(M,[g]_G)$ is attained and the Yamabe equation admits a smooth positive $G-$invariant solution (see for instance \cite{Hebey-Vaugon2}). On the other hand, if there exists a point $P$ that belongs to a minimal finite orbit such that $\omega(P)\leq \frac{n-6}{2}$, then by \eqref {HebeyVaugonconjw} the $G-$equivariant Yamabe constant is attained too.   

Let $u\in H^2_1(M)-\{0\}$ and $V\in Gr^2(H_1^2(M))$. In order to simplify the notation in the proof of Theorem \ref{strictinequalityA},  we define the following functional $$H_u: V-\{0\}\longrightarrow \re,\, H_u(v):=\frac{\int_M a_n|\nabla v|^2_g+s_gv^2dv_g}{\int_Mu^{p_{n}-2}v^2dv_g}.$$

In order to prove the theorem, using Proposition  \ref{variationalcharact}, it is sufficient to find $u\in C^{\infty}_{>0,G}(M)$ and $V\in Gr^2(H_{1,G}^2(M))$ such that 
 $$\sup_{v\in V-\{0\}}H_u(v)(\int_Mu^{p_n}dv_g)^{\frac{2}{n}}<  \Big(Y_G(M,[g]_G)^{\frac{n}{2}} +Y(S^n)^{\frac{n}{2}}\Lambda_G\Big)^{\frac{2}{n}}.$$

 If $\Lambda_G=+\infty$, then \eqref{E2ndYbounds} is trivially satisfied. Hence, we can assume that $G$  has   a finite orbit on $M$.  Let $O_G(P)=\{P_1,\dots,P_k\}$ be a finite minimal orbit, with $P_1:=P$. Let $\varphi\in C^{\infty}_{>0,G}(M)$ be a minimizing  function of the Yamabe functional, which realizes the $G-$equivariant Ya\-ma\-be constant.  We normalize $\varphi$, such that $\|\varphi\|_{p_n}=1$.

Let us consider the following $G-$invariant function:
 \begin{equation}\label{2ndfunctiontest1}
 u_{\varepsilon}:=\begin{cases}
  J(\psi_{\varepsilon})^{\frac{n-2}{4}}{\psi}_{\varepsilon}+Y_G(M,[g]_G)^{\frac{n-2}{4}}\varphi, & {\rm if }\ Y_G(M,[g]_G)>0,\\
  {\psi}_{\varepsilon}, & {\rm if }\ Y_G(M,[g]_G)=0,
 \end{cases}
 \end{equation}
where   $\psi_{\varepsilon}$ is the test function defined in  \eqref{testfunctionw}, that we normalize, in order to have $\|\psi_{\varepsilon}\|_{p_n}=1$.
 
Let $V_{\varepsilon}\in Gr^2(C^\infty_G(M))$ be a $2-$dimensional subspace of $C^\infty_G(M)$, defined by 
\begin{equation}\label{subspacetest}
 V_{\varepsilon}:=span\big(\psi_{\varepsilon}, \varphi\big).
\end{equation}

\paragraph{\bf The case $\boldsymbol{Y_G(M,[g]_G)>0}$}  Assume that $Y_G(M,[g]_G)>0$ and for $v=\alpha \psi_{\varepsilon} +\beta \varphi\in V_{\varepsilon}$, with $(\alpha,\beta)\in\re^2-\{(0,0)\}$, we have that  
$$
  H_{u_\varepsilon}(v)=\frac{\alpha^2J(\psi_{\varepsilon})+\beta^2Y_G(M,[g]_G)+2\alpha\beta Y_G(M,[g]_G)\int_M|\varphi|^{p_n-2}\varphi \psi_{\varepsilon}dv_g}{\alpha^2\int_M|u_{\varepsilon}|^{p_n-2}\psi_{\varepsilon}^2dv_g+\beta^2\int_M|u_{\varepsilon}|^{p_n-2}\varphi^2dv_g+2\alpha\beta\int_M|u_{\varepsilon}|^{p_n-2}\psi_{\varepsilon}\varphi dv_g}.$$
 
Therefore, 
 \begin{equation}\label{bound}H_{u_\varepsilon}(v)\leq \frac{\alpha^2J(\psi_{\varepsilon})+\beta^2Y_G(M,[g]_G)+2\alpha\beta Y_G(M,[g]_G)\int_M|\varphi|^{p_n-2}\varphi \psi_{\varepsilon}dv_g}{\alpha^2J(\psi_{\varepsilon})+\beta^2Y_G(M,[g]_G)+2\alpha\beta\int_M|u_{\varepsilon}|^{p_n-2}\psi_{\varepsilon}\varphi dv_g}.\end{equation}

 Following closely the proof  of Theorem 5.4 in \cite{Ammann-Humbert}, for $n>6$, we can see that for any $v\in V_{\varepsilon}-\{0\}$, we have:
 $$H_{u_{\varepsilon}}( v)\leq 1+O(\varepsilon^{\frac{n-2}{2}} ).$$ 

It is known, see for instance Lemma 5.7 in \cite{Ammann-Humbert},  that if $q>2$ then, there exists $C>0$ such that $(a+b)^{q}\leq a^q+b^q+C(a^{q-1}b+ab^{q-1})$, for any $a,b\geq 0$.
Hence,  
\begin{multline*}
\int_Mu_{\varepsilon}^{p_n}dv_g\leq J(\psi_{\varepsilon})^{\frac{n}{2}}\int_M\psi_{\varepsilon}^{p_n}dv_g+Y_G(M,[g]_G)^{\frac{n}{2}}\int_M\varphi^{p_n}dv_g\\
+C(\int_M\varphi^{p_n-1}\psi_{\varepsilon}dv_g+\int_M\psi_{\varepsilon}^{p_n-1}\varphi dv_g)\\
=J(\psi_{\varepsilon})^{\frac{n}{2}}+Y_G(M,[g]_G)^{\frac{n}{2}}+C(\int_M\varphi^{p_n-1}\psi_{\varepsilon}dv_g+\int_M\psi_{\varepsilon}^{p_n-1}\varphi dv_g). 
\end{multline*}
 
By \eqref{estimationpowerstest}, we have that 
 \begin{equation}\label{volumestimation}
 \int_Mu_{\varepsilon}^{p_n}dv_g\leq J(\psi_{\varepsilon})^{\frac{n}{2}}+Y_G(M,[g]_G)^{\frac{n}{2}}+O(\varepsilon^{\frac{n-2}{2}}).
 \end{equation}
On the other hand, by the estimation \eqref{estimationAubin}, we have that   $$J(\psi_{\varepsilon})= Y(S^n)\Lambda_G^{\frac{2}{n}}-A\varepsilon^4+o(\varepsilon^4),$$
with $A>0$.
 Therefore, $$J(\psi_{\varepsilon})^{\frac{n}{2}}\leq Y(S^n)^{\frac{2}{n}}\Lambda_G-A\varepsilon^4+o(\varepsilon^4).$$ 
 and  $$\int_Mu_{\varepsilon}^{p_n}dv_g \leq Y(S^n)^{\frac{2}{n}}\Lambda_G+Y_G(M,[g]_G)^{\frac{n}{2}}-A\varepsilon^4+o(\varepsilon^4)+O(\varepsilon^{\frac{n-2}{2}}).$$
 Finally, we get that
 $$\Big(\int_Mu_{\varepsilon}^{p_n}dv_g\Big)^{\frac{2}{n}}\leq \Big(Y(S^n)^{\frac{n}{2}}\Lambda_G+Y_G(M,[g]_G)^{\frac{n}{2}}\Big)^{\frac{2}{n}}-\tilde{C}\varepsilon^4+O(\varepsilon^\frac{n-2}{2})+o(\varepsilon^4),$$
 with $\tilde{C}>0$.
If $n\geq 11$ and $\varepsilon$ is small enough, we obtain that
\begin{multline*}\sup_{v\in V_{\varepsilon}-\{0\}}H_{u_{\varepsilon}}(v)(\int_Mu_{\varepsilon}^{p_n}dv_g)^{\frac{2}{n}}
\leq \big(1+O(\varepsilon^{\frac{n-2}{2}})\big)\Big[\Big(Y(S^n)^{\frac{2}{n}}\Lambda_G+Y_G(M,[g]_G)^{\frac{n}{2}}\Big)^{\frac{2}{n}}\\
-\tilde{C}\varepsilon^4+O(\varepsilon^\frac{n-2}{2})\Big]
 <\Big(Y(S^n)^{\frac{2}{n}}\Lambda_G+Y_G(M,[g]_G)^{\frac{n}{2}}\Big)^{\frac{2}{n}}.
 \end{multline*}

\paragraph{\bf The case $\boldsymbol{Y_G(M,[g]_G)=0}$} For $n>4$, let $\alpha_{\varepsilon}$ and $\beta_{\varepsilon}$ be such that $\alpha^2_{\varepsilon}+\beta_{\varepsilon}^2=1$ and 
 $\sup_{v\in V_{\varepsilon}-\{0\}}H_{u_{\varepsilon}}(v)=H_{u_{\varepsilon}}(\alpha_{\varepsilon}\psi_{\varepsilon}+\beta_{\varepsilon}\varphi)$. If $\alpha_{\varepsilon}=0$, then $$\sup_{v\in V_{\varepsilon}-\{0\}}H_{u_{\varepsilon}}(v)=0.$$
 If $\alpha_{\varepsilon}\neq 0$, then applying \eqref{estimationpowerstest} we get that
 $$\sup_{v\in V_{\varepsilon}-\{0\}}H_{u_{\varepsilon}}(v)=
 \frac{J(\psi_{\varepsilon})}{1+2\big(\frac{\beta_{\varepsilon}}{\alpha_{\varepsilon}}\big)O(\varepsilon^{\frac{n-2}{2}})+\big(\frac{\beta_{\varepsilon}}{\alpha_{\varepsilon}}\big)^2O(\varepsilon^2)}.$$
 Since $\|u_{\varepsilon}\|_{p_n}=\|\psi_{\varepsilon}\|_{p_n}=1$,  by \eqref{estimationAubin} (for $n>6$) it follows that 
  $$\sup_{v\in V_{\varepsilon}-\{0\}}H_{u_{\varepsilon}}(v)\big(\int_Mu_{\varepsilon}^{p_n}dv_g\big)^{\frac{2}{n}}=\frac{Y(S^n)\Lambda_G^{\frac{2}{n}}-A\varepsilon^4+o(\varepsilon^4)}{1-O(\varepsilon^{n-4})},$$
 with $A>0$. Therefore, if $n-4>4$, for $\varepsilon$ small enough, we have
 $$\sup_{v\in V_{\varepsilon}-\{0\}}H_{u_{\varepsilon}}(v)\big(\int_Mu_{\varepsilon}^{p_n}dv_g\big)^{\frac{2}{n}}<Y(S^n)\Lambda_G^{\frac{2}{n}}$$
 and the theorem follows  \hfill $\square$
 

\begin{proof}[Proof of Proposition \ref{boundsproposition}]

The right hand side inequality is trivially satisfies if $\Lambda_G=+\infty$. So,  we assume that $\Lambda_G<\infty$. Let us consider the test function defined by    
 $$w_{\varepsilon}:=J(\phi_{\varepsilon})^{\frac{n-2}{4}}{\phi}_{\varepsilon}+Y_G(M,[g]_G)^{\frac{n-2}{4}}\varphi,$$ where $\phi_{\varepsilon}$ is as in \eqref{asymptotically},  $\varphi$ is a $G$-invariant function that realizes $Y_G(M,[g]_G)$ with $\|\varphi\|_{p_n}=1$, and $V_{\varepsilon}:=span (\phi_\e,\varphi)$. 
 By \eqref{volumestimation} we have that 
 \begin{equation}
  \lim_{\varepsilon \to 0}\big(\int_Mw_{\varepsilon}^{p_n}dv_g\big)^{\frac{2}{n}}=\Big(Y_G(M,[g]_G)^{\frac{n}{2}}+Y(S^n)^{\frac{n}{2}}\Lambda_G\Big)^{\frac{2}{n}}.
 \end{equation}
Therefore, it follows by \eqref{bound} that 
$$\limsup_{\varepsilon \to 0}\sup_{v\in V_{\varepsilon}-\{0\}}H_{w_{\varepsilon}}(v)\big(\int_Mw_{\varepsilon}^{p_n}dv_g\big)^{\frac{2}{n}}\leq \Big(Y_G(M,[g]_G)^{\frac{n}{2}}+Y(S^n)^{\frac{n}{2}}\Lambda_G\Big)^{\frac{2}{n}},$$
which implies that $$Y^2_G(M,[g]_G)\leq \Big(Y_G(M,[g]_G)^{\frac{n}{2}}+Y(S^n)^{\frac{n}{2}}\Lambda_G\Big)^{\frac{2}{n}}.$$

In  order to prove the lower bound $$2^\frac{2}{n}Y_G(M,[g]_G)\leq Y^2_G(M,[g]_G),$$ 
it is sufficient to show that 
\begin{equation}\label{lowerbound}
2^\frac{2}{n}Y_G(M,[g]_G)\leq \sup_{v\in V-\{0\}}H_{u}(v),
\end{equation}
for any $u\in C^{\infty}_{>0,  G}(M)$, $\|u\|_{p_n}=1$, and $V$ any 2-dimensional subspace of $H^2_{1, G}(M)$.   
The arguments used  in \cite[Proposition 5.6]{Ammann-Humbert} to prove the same estimate work  without any
signi\-fi\-cant changes to prove \eqref{lowerbound}. 
\end{proof}

\section{Nodal solutions with symmetry}\label{sectionnodalsolutions}

In \cite{Ammann-Humbert},  it was proved that if $Y^2(M,[g])$ is positive  and   attained, then there exists a nodal solution of the Yamabe equation.   
In this section,  we will prove Theorem \ref{nodalsolution} that asserts a similar fact, in the equivariant setting.
In the following, we briefly comment  what remains true, without any significant changes, in the $G-$equivariant second Yamabe context from \cite{Ammann-Humbert}.

Assume that $Y(M,[g])>0$.  Then,  the first eigenvalue $\lambda_{1}(g)$ of $L_g$ is positive and this implies that the conformal Laplacian is a coercive operator. In, particular $L_g$ is invertible. Let $G$ be a closed subgroup of $I(M,g)$.  For $u\in L^{p_n}_{\geq 0,G}(M)$, we consider the generalized $G-$invariant metric $g_u=u^{p_n-2}g$. 

By the standard variational method, we claim that there exist $v_1$ and  $v_2$ that belong to $H^2_{1,G}(M)$, which satisfy, in a weak sense, the following linear equations

\begin{gather}
L_g(v_1)=\lambda_{1,G}(g_u)u^{p_n-2}v_1,\label{v_1}
\\
L_g(v_2)=\lambda_{2,G}(g_u)u^{p_n-2}v_2,\label{v_2}
\end{gather} 
 and also satisfy
\begin{equation}\label{orth}
\int_Mu^{p_n-2}v_iv_jdv_g=\delta_{ij}.
\end{equation}

Actually, the subspace $V=span(v_1,v_2)$ realizes the infimum in the va\-ria\-tional characterization of $\lambda_{2,G}(g_u)$. More precisely, we have that $$\lambda_{2,G}(g_u)=\frac{\int_M a_n|\nabla v_2|^2_g+s_gv^2_2dv_g}{\int_Mu^{p_{n}-2}v^2_2dv_g}.$$

Suppose that $Y^2_G(M,[g]_G)>0$ is attained by the generalized metric $g_u$. Since we can assume that $\|u\|_{p_n}=1$, then  
$Y^2_G(M,[g]_G)=\lambda_{2,G}(g_u)$. Let $v_2$ as in \eqref{v_2}. Following closely the proof of Theorem 3.4 in \cite{Ammann-Humbert},  it can be proven that
$v_2$ is a function that changes sign. Indeed,  this is a consequence of the inequality $\lambda_{1,G}(g_u)<\lambda_{2,G}(g_u)=Y^2_G(M, [g]_G)$.

The key argument to prove the above facts is the following:  if  $w_1$, $w_2 \in H^2_{1,G}(M)-\{0\}$ are non-negative functions such that the set $\{w_1\neq 0\}\cap \{w_2\neq 0\} $  has measure $0$ and satisfy the following inequalities
\begin{gather}\int_M a_n|\nabla w_1|^2+s_gw_1dv_g\leq \ Y^2_G(M,[g]_G)\int_Mu^{p_n-2}w_1^2dv_g,\label{LIY1}
\\
\int_M a_n|\nabla w_2|^2+s_gw_2dv_g\leq \ Y^2_G(M,[g]_G)\int_Mu^{p_n-2}w_2^2dv_g,\label{LIY2}
\end{gather}
  then, $u\in span(w_1,w_2)$ and the equalities hold in \eqref{LIY1} and \eqref{LIY2} (see the details in \cite[Lemma~3.3]{Ammann-Humbert}).
Since $w_1=\sup (0,v_2)$ and $w_2=\sup (0,-v_2)$ satisfy \eqref{LIY1} and \eqref{LIY2}, it follows that there exist $a,\ b>0$, such that 
\begin{equation}\label{ee} 
u=a\sup (0,v_2)+b\sup(0,-v_2)
 \end{equation}

Then,  $v_2\in L^{p_n+\gamma}_G(M)$ for some $\gamma>0$ (see \cite[Lemma~3.1]{Ammann-Humbert}). Using a standard bootstrap argument  it can be proved that $v_2\in C^{2,\alpha}_G(M)$ for all $\alpha \in  (0,1)$. Therefore, by Equality \eqref{ee} we have that $u\in C^{0,\alpha}_G(M)$.

The same arguments used in the proof of Theorem 3.4 in \cite{Ammann-Humbert} can be used to prove that for any $h\in C^{\infty}_G(M)$ with $supp(h)\subset supp (u)$ the following equality holds 
\begin{equation}\label{condition}
\int_Mu^{p_n-3}v_2^2hdv_g=\int_Mu^{p_n-1}hdv_g.
\end{equation}

\begin{proof}[Proof of Theorem \ref{nodalsolution}]
 It is  sufficient to show that $u=|v_2|$. 
 Suppose that there exists $P_0\in M$, such that \begin{equation}\label{conditionpunt}u^{p_n-3}(P_0)v_2^2(P_0)> u^{p_n-1}(P_0).\end{equation}

Assume that the orbit  $O_G(P_0)$ is finite. For $\delta$ small enough, let us consider the $G-$equi\-variant function $\phi_{\varepsilon}$ defined as in \eqref{equivtestfunction1}, centered on the orbit of $P_0$, such that $supp (\phi_{\varepsilon})\subset supp (u)$. Since $u$ is a continuous function, for $\delta$ small enough, we have that 
$$u^{p_n-3}(Q)v_2^2(Q)\phi_{\varepsilon}(Q)> u^{p_n-1}(Q)\phi_{\varepsilon}(Q)$$ for any  $Q$ that belongs to the interior of  $supp(\phi_{\varepsilon})$.
Therefore,   we have that
$$\int_Mu^{p_n-3}v_2^2
\phi_{\varepsilon}dv_g>
\int_Mu^{p_n-1}\phi_{\varepsilon}dv_g$$
which, by \eqref{condition},  is a contradiction.

Now, let us assume that the orbit of $P_0$  is not finite. We know, since $G$ is a compact subgroup, that $O_G(P_0)$ is an embedded submanifold  of $M$, with $\dim O_G(P_0)\geq 1$. By the slice Theorem, given $\varepsilon>0$ small enough,  for any $P\in O_G(P_0)$ there exists a slice $$\Sigma_P:=\exp_P(\{z\in \nu(O_G(P_0)): |z|\leq \varepsilon\}),$$
such that for any $\sigma\in G$ $$\sigma.\Sigma_P=\Sigma_{\sigma . P}.$$

Then,  we can define $\zeta_\varepsilon$ a $G-$invariant function, with support in a tubular neighborhood $T_{\varepsilon}$ of $O_G(P_0)$ such that   $$u^{p_n-3}(Q)v_2^2(Q)\zeta_{\varepsilon}(Q)> u^{p_n-1}(Q)\zeta_{\varepsilon}(Q)$$ for any $Q\in T_{\varepsilon}$. Hence,
$$\int_Mu^{p_n-3}v_2^2\zeta_{\varepsilon}dv_g>\int_Mu^{p_n-1}\zeta_{\varepsilon}dv_g,$$  
which is again a  contradiction.
Of course, the same happens if we assume the   opposite inequality in \eqref{conditionpunt}. 
Therefore, $$u^{p_n-3}v_2^2= u^{p_n-1},$$
which implies that $u=|v_2|$.
\end{proof}

\begin{Remark} Let $(M,g)$ be a closed  Riemannian manifold with $\dim(M)\geq 3$ and $G$ be a closed subgroup of $I(M,g)$. If $Y^2_G(M,[g]_G)=0$, then we have that $Y_G(M,[g]_G)=Y(M,g)<0$. Indeed, 
 since $ Y_G(M,[g]_G)\leq Y^2_G(M,[g]_G)$, we have that $Y_G(M,[g]_G)\leq 0$, and therefore, $Y_G(M,[g]_G)=Y(M,[g])$ and the $G-$equivariant Yamabe constant is attained. It follows by definition of the second Yamabe constant that 
$$Y(M,[g])=Y_G(M,[g]_G)\leq Y^2(M,[g])\leq Y^2_G(M,[g]_G)=0.$$ 
 Let us assume that $Y(M,[g])=0$. Then, $\lambda_1(g)=0$  and $Y^2(M,[g])=0$. By \cite[Theorem 1.4]{Ammann-Humbert}, $Y^2(M,[g])$ is attained by a genelarized metric $h=u^{p_n-2}g$ with $u\in L^{p_n}_{\geq 0}(M)$. Hence, $\lambda_2(h)=0$. But this implies that $\lambda_2(g)=0$ wich is a contradiction, since $g$ is a Riemannian metric and therefore we have the strict inequality $\lambda_1(g)<\lambda_2(g)$.      
Then, $Y(M,[g])=Y_G(M,[g]_G)$ is a negative constant. 
 
\end{Remark}

\begin{Proposition}\label{nodalY2=0} Let $(M,g)$ be a closed Riemannian manifold of dimension $n\ge 3$ and $G$ be a closed subgroup of $I(M,g)$. Assume that 
$Y^2_G(M,[g]_G)=0$ and is attained. Then,
 there exists a $G-$invariant nodal solution of the Yamabe equation 
\begin{equation}\label{eigenfuncionzero}
L_g(v)=0.\end{equation}
 \end{Proposition}

 \begin{proof} Let $u_0\in L^{p_n}_{\geq 0, G}$ such that $\lambda_{2,G}(g_{u_0})=0$. Then, there exists $V_0\in Gr_{u}^2(H^2_{1,G}(M))$ such that $$\sup_{v\in V_0-\{0\}}H_{u_0}(v)=0.$$ 
This implies that $$0=\sup_{v\in V_0-\{0\}}H_1(v)\geq \lambda_{2,G}(g).$$
 By hypothesis, $\lambda_{2,G}(h)\geq 0$ for any generalized metric $h$, hence $\lambda_{2,G}(g)=0$. If we consider  an eigenfunction $v$ associated to the eigenvalue $\lambda_{2,G}(g)=0$, then $v$ satisfies Equation \eqref{eigenfuncionzero} and is a  sign changing function.    
\end{proof}

If $(M,g)$ satisfies the assumptions of Theorem \ref{strictinequalityA}, then as a consequence of the proposition above and Theorem \ref{Y2attained} (see Section \ref{SectproofY2attained}) we have a $G$-invariant nodal solution of the Yamabe equation (Corollary \ref{CorolNodalsolution}). 

  \section{Riemannian products}

\begin{proof}[Proof of Theorem \ref{asyntotic}]
 
The proof of Theorem  1.1  in \cite{A-F-P} works in the equivariant case. In this setting we have the action of the group $G$, however this action is the trivial one in $(N,h)$. Therefore, the proof of the  above mentioned theorem can be adapted to prove \eqref{itema}. For the same reason, Equality \eqref{itemb} follows from the proof of Theorem 1.1 in \cite{Henry}.
\end{proof}

Let $(M,g)$ and $G$ be a subgroup of $I(M,g)$ as in Theorem \ref{TC}.  Note that any  function in $C^{\infty}_0(M\times \re^l)$, which is constant on $M$,  is automatically $G-$in\-va\-riant, since $G$ acts trivially on $\re^l$. Therefore, 
$$Y_G(M\times \re^l,[g+g^l_e])\leq Y_{\re^l}(M\times \re^l,g+g^l_e),$$
where $Y_{\re^l}(M\times \re^l,g+g^l_e)$ is the $\re^l-$ Yamabe constant of $(M\times \re^l,g+g^l_e)$, defined by
$$Y_{\re^l}(M\times \re^l,g+g^l_e):=\hspace{-0.3cm}\inf_{u\in C^{\infty}_{0}(\re^l)-\{0\}}\hspace{-0.2cm}\frac{\int_{\re^l} a_{n+l}|\nabla u|_{g^l_0}^2+s_gu^2dv_{g^l_0}}{\|u\|_{p_{n+l}}^2} \big(vol(M,g)\big)^{\frac{2}{n+l}}.$$

If in addition, $s_g$ is constant and $vol(M,g)=1$, then by Theorem 1.4 of \cite{A-F-P} it follows that 
$$Y_G(M\times \re^l,[g+g^l_e]_G)\leq C(n,l)s_g^{\frac{n}{n+l}},$$
where $C(n,l)=a_{n+l}^{\frac{l}{n+l}}(n+l)l^{-\frac{l}{n+l}}n^{\frac{-n}{n+l}}\alpha_{n,l}^{-1}$ with
$\alpha_{n,l}$ the $(n,l)$ Niren\-berg-Gagliardo constant. Hence, if $s_g$ is small enough we have that
$$Y_G(M\times \re^l,[g+g^l_e]_G)\leq Y(S^{n+l})\Lambda_{G}^{\frac{2}{n+l}}.$$

If, in addition $l\geq 2$, Theorem \ref{TC} says that we have the strict inequality in the inequality above.

To prove Theorem \ref{TC}, we will need  a $G-$e\-qui\-va\-riant conformal normal coordinate system at $(P,0)\in (M\times \re^l)$. That is to say, $h$ is a $G-$in\-va\-riant Riemannian metric, that belongs to $[g+g^l_e]$ 
such that in a normal coordinate system we have:

\begin{itemize}
 \item If $n+l=4$, $\det(h)=1+O(r^3)$, (which is equivalent to have $Ricc_h(P,0)(h)=0$).
  \item If $n+l=5$, $\det(h)=1+O(r^4)$, (which is equivalent to have $Ricc_h(P,0)=0$ and $\nabla Ricc_h(P,0)=0$).
 \item If $n+l\geq 6$, For any $s>>1$,  $\det(h)=1+O(r^s)$ and $Ricc_h(P,0)=0$. 
\end{itemize}

In \cite{Hebey-Vaugon}, Hebey and Vaugon proved the existence of these kind of metrics for  the compact setting.  Here, $M\times \re^l$ is not compact, but the action of $G=G_1\times  \{Id_{\re^l}\}$ is trivial in $(\re^l, g_e^l)$. When $G$ is the trivial group of $I(M\times \re ,g+g^l_e)$, Inequality \eqref{TCineq} is proved by Akutagawa, Florit, and Petean in \cite{A-F-P}.

\begin{proof}[Proof of Theorem \ref{TC}]
We can assume that the action of $G_1$ on $M$ has finite orbits. If it does not, the inequality of the theorem holds trivially. Let $O_G(Q_1)=\{Q_1 =(P_1,0),\dots, Q_k=(P_k,0)\}$ be a minimal finite orbit.
Assume that $n+l\geq 6$. Note that $(M\times\re^l,g+g^l_0)$ is not locally conformally flat. Indeed, when both $\dim(M)=n$ and $\dim(N)=l$ are greater or equal than 2, the Riemannian product $(M\times N,g+h)$ is locally conformally flat if and only if  $(M,g)$ and $(N,h)$ have constant sectional curvature $c$ and $-c$, respectively (see \cite{Yau}).  By assumption, $s_g>0$. Hence, if  $(M,g)$ has constant sectional curvature it must be positive. Therefore,  $(M\times\re^l,g+g^l_0)$ is not locally conformally flat, which is equivalent to say that the  Weyl tensor $W_{g+g^l_e}$ never vanish completely. More precisely, since for any $P\in M$ $s_g(P)>0$,  then 
$W_{g+g^l_e}(P,0)\neq 0.$
Therefore, $\omega(Q_j)=0$ for any point $Q_j$ in the minimal orbit $O_G(Q_1)$.
The $G-$equivariant Yamabe constant is an invariant of $[g+g^l_e]_G$, then in order to prove the theorem it will be convenient to consider 
a $G-$equivariant conformal normal coordinate system $h$ instead of $g+g^l_e$.  Since the Weyl tensor is invariant by 
a conformal change of the metric,   then $\omega_h(Q_j)= 0$, for any $j=1,\dots,k$.


Let us consider, for small $\delta>0$,   the $G-$equivariant function $\phi_{\varepsilon}$ defined in \eqref{testfunction1}.  We extend $\phi_\e$ trivially to $M\times \re^l$. By \eqref{estimationAubin}, for $\varepsilon$ small enough,  the following inequality holds for the Yaamabe functional of $M\times \re^l$: 
  $$Y(\phi_{\varepsilon})<k^{\frac{2}{n+l}}Y(S^{n+l}),$$
see for instance \cite{Hebey-Vaugon}.
Assume now that $n+l=4,5$. Akutagawa, Florit and Petean proved in  \cite{A-F-P} that for any $\bar{P}\in M$ there exists a unique normalized Green function $G_{\bar{Q}}$ for the conformal Laplacian $L_{g+g^l_e}$ with pole at $\bar{Q}=(\bar{P},0)$. Also, they proved that $(M\times \re^l-\{\bar{Q}\},G_{\bar{Q}}^{\frac{4}{n+l-2}}(g+g^l_e))$, which is  a scalar-flat and asymptotically flat manifold, has positive mass. Then, after choosing a $G-$equivariant conformal normal coordinate system, we  use the Schoen's test function in order to construct an appropriate $G-$invariant test function. 
Let

$$\zeta_{\bar{Q}, \varepsilon}(Q):=
\begin{cases}
 \Big(\frac{\varepsilon}{\varepsilon^2+r^2(Q)}\Big)^{\frac{n+l-2}{2}} & \mbox{if }\  r(Q)\leq \delta,\\ 
   \varepsilon_0 \Big(G_{\bar{Q}}(Q)-\eta(Q)\alpha_{\bar{Q}}(Q)\Big)& \mbox{if }\   \delta\leq r(Q)\leq  2 \delta, \\
   \varepsilon_0 G_{\bar{Q}}(Q) &\mbox{if}\ r(Q)>2\delta,\end{cases}$$
where $r(Q):=d(\bar Q,Q)$, $\eta$ is a cut-off function as in \eqref{testfunction1},  $\varepsilon_0>0$ is a small constant  that satisfies
 $$\varepsilon_0=\Big[\frac{\varepsilon}{(\delta^{2-(n+l)}+A)^{\frac{2}{n+l-2}}(\varepsilon^2+\delta^2)}\Big]^{\frac{n+l-2}{2}},$$
 and $\alpha_{\bar{Q}}$ is the function that appears in asymptotic expansion of the Green function $G_{\bar{Q}}$ (see for instance   \cite{LP} and \cite{Schoen1}). Then,   we consider the $G-$equiva\-riant function in $M\times \re^l$ defined by 

$$\zeta_{\varepsilon}(Q):=\sum_{i=1}^k\zeta_{Q_i, \varepsilon}(Q).$$
By computations made in  \cite{Hebey-Vaugon} we get that for $\varepsilon$ small enough  
$$Y(\zeta_{\varepsilon})<Y(S^{n+l})k^{\frac{2}{n+l}},$$
and this proves the theorem.
\end{proof}

\begin{proof}[proof of Corollary \ref{strictineqforproduct}] By Theorem \ref{TC} we have that 
$$2^{\frac{2}{k}}Y_G(M\times \re^n,[g+g^l_e]_G)<\Big(Y_G(M\times \re^l,[g+g^l_e]_G)^{\frac{k}{2}} +Y(S^{k})^{\frac{k}{2}}\Lambda_G\Big)^{\frac{2}{k}},$$
with $k=n+l$.
Therefore, for $t$ large enough, by Theorem \ref{asyntotic}, we get the desired inequality.
\end{proof}

Corollary \ref{nodalsolproduct} is an immediate  consequence of the corollary above, Theorem \ref{Y2attained}, and Theorem \ref{nodalsolution}. 

\section{The subcritical case}\label{subcrisection}

Let $(M^n,g)$ and $(N^l,h)$ be closed Riemannian manifolds  of constant scalar curvature such that $s_{g+h}>0$. Assume that  $n+l\geq 3$. Let $G=G_1\times G_2$ be a compact subgroup of $I(M\times N, g+h)$ where   $G_1$ and $G_2$ are closed subgroups of $I(M,g)$ and $I(N,h)$, respectively. We also assume that the  action of $G_1$ on $M$ is not transitive.

When $G_1=\{Id_{M}\}$, Proposition \ref{subcritical} and Corollary \ref{subcriticalnodal} was proved by Petean in \cite{Peteannodal}. The key argument to prove this case is that $H^2_1(M)$ is compactly embedded  in $L^{p_{n+l}}(M)$, which is a consequence of the Rellich-Kondrakov theorem  since $p_{n+l}<p_{n}$ (see for instance \cite{LP}).

When $G_1$ is not the trivial group the situation is similar because the  inclusion of $H^2_{1,G}(M)$ in $L^{p_{n+l}}(M)$ is a compact operator (see \cite{Hebey-Vaugon3}).

In the following we sketch the proof  of Proposition \ref{subcritical}.

We can assume without loss of generality that $vol(N,h)=1$.
Let $g_{u_k}=u_k^{p_{n+l}}(g+h)$ with $u_k\in C^{\infty}_{G_1,>0}(M)$  and ${\|u_k\|_{p_{n+l}}}=1$ be  a minimizing sequence of $Y_{M,G_1}^2(M\times N,g+h)$. Let $\{v_1^k\}$ and $\{v_2^k\}$ 
as in \eqref{v_1} and \eqref{v_2}, respectively. The sequence $\{u_k\}$ is bounded in $L^{p_{n+l}}(M)$, therefore there exists a subsequence that converges weakly to $u$. On the other hand, $\{v_1^k\}$ and $\{v_2^k\}$ are bounded sequences in   $H^2_{1,G_1}(M)$. Hence, there exist $v_1$ and $v_2$ such that $v_k^1\rightharpoonup v_1 $ and  $v_k^2\rightharpoonup v_2 $ weakly in $H^2_{1,G_1}(M)$. Since $H^2_{1,G_1}(M)\subseteq L^{p_{n+l}}(M)$, we have that 
there exists a subsequence of $v_k^1$ and $v_k^2$ that converge strongly in $L^{p_{n+l}}(M)$. Then, we can pass into the limits in \eqref{orth} and we obtain that 
\begin{equation}\label{dim2}\int_Mu^{p_{n+l}}v_iv_j=\delta_{ij}.\end{equation}

Furthermore, $u$, $v_1$, and $v_2$ satisfy  in the sense of distributions that
\begin{equation}\label{E1}L_g(v_1)=\Big(\limsup_{k \to \infty}\lambda_{1, G_1}^M(g_{u_k})\Big)u^{p_{n+l}-2}v_1,
 \end{equation}
and
\begin{equation}\label{E2}L_g(v_2)=Y^2_{M,G_1}(M\times N, g+h)u^{p_{n+l}-2}v_2.
\end{equation}

By \eqref{dim2}, we have that $V_0:=span(v_1,v_2) \in Gr_u^2(H^2_{1,G_1}(M))$. Since \linebreak$\limsup_{k}\lambda_{1, G_1}^M(g_{u_k})\leq Y^2_{M,G_1}(M\times N, g+h)$,  by \eqref{E1} and \eqref{E2} we have that 
$$ sup_{v\in V_0-\{0\}}H_{u}(v)\big(\int_M u^{p_{n+l}}dv_g\big)^{\frac{2}{n+l}}\leq Y^2_{M,G_1}(M\times N, g+h).$$

Hence, $Y^2_{M,G_1}(M\times N, g+h)$ is achieved by the generalized metric $g_u$. In the variational characterization of the $G_1-$equivariant $M-$second Yamabe constant, the infimum is realized by the subspace $V_0$, and the supremum in $V_0$ is achieved by the $G_1-$invariant function $v_2$.

Corollary \ref{subcriticalnodal} follows by similar arguments as the ones used to proved Theorem \ref{nodalsolution} (see Section \ref{sectionnodalsolutions}).

\section{Proof of Theorem \ref{Y2attained} }\label{SectproofY2attained}

In order to prove Theorem \ref{Y2attained} we will need the following result due to Hebey and Vaugon (see \cite{Hebey}):

\begin{Theorem}\label{HVsob} Let $(M,g)$ be a closed Riemannian manifold of dimension $n\ge 3$ and $G$ be a closed subgroup of $I(M,g)$, with $\Lambda_G<\infty$. Then, there exists $B_0\in\re$, such that 
for any $v\in H^2_{1,G}(M)-\{0\}$, we have that
\begin{equation} Y(S^n)\Lambda_G^{\frac{2}{n}}\leq  \frac{\int_M a_n|\nabla v|^2_g+B_0v^2dv_g}{\|v\|_{p_n}^2}
\end{equation} 
\end{Theorem}

To prove Theorem \ref{Y2attained}, we follow closely the  proof of Theorem 1.4 in \cite{Ammann-Humbert},  pointing out what  we should adapt  to the equivariant setting.
As in the Section \ref{subcrisection}, 
let $g_{u_k}=u_k^{p_{n}}g$ with $u_k\in C^{\infty}_{G,>0}(M)$  and ${\|u_k\|_{p_{n}}}=1$ be  a minimizing sequence of $Y_{G}^2(M,[g]_G)$. Let $\{v_1^k\}$ and $\{v_2^k\}$ as in \eqref{v_1} and \eqref{v_2}, respectively. These functions satisfies in a weak sense the following:
\begin{gather}
L_g(v_1^k)=\lambda_{1,G}(g_{u_k})u^{p_n-2}_kv_1^k,
\\
L_g(v_2^k)=\lambda_{2,G}(g_{u_k})u^{p_n-2}_kv_2^k,\label{v_2k}
\\
\int_Mu^{p_n-2}v_iv_jdv_g=\delta_{ij}.
\end{gather} 

 The sequence $\{u_k\}$ is bounded in $L^{p_{n}}(M)$, therefore there exists a subsequence that converges to $u$ weakly in $L^{p_n}(M)$. On the other hand, $\{v_1^k\}$ and $\{v_2^k\}$ are bounded sequence in   $H^2_{1,G}(M)$, therefore, there exist $v_1$ and $v_2$ such that $v_k^1\rightharpoonup v_1 $ and  $v_k^2\rightharpoonup v_2 $ weakly in $H^2_{1,G}(M)$. Then (up to a subsequence) we have in a weak sense 
\begin{gather*}
L_g(v_1)=\big(\lim_{k\to +\infty} \lambda_{1,G}(g_{u_k})\big) u^{p_n-2}v_1,\\
L_g(v_2)=Y^2_G(M,[g]_G) u^{p_n-2}v_2.
\end{gather*}
However, it is not clear that $u^{p_n-2}v_1$ and $u^{p_n-2}v_2$ are linearly independent in $H^2_{1,G}(M)$. If it is the case, then they span a $2-$dimensional subspace of $H^2_{1,G}(M)$. It yields that $u^{p_n-2}g$ realizes $Y^2_G(M,[g]_G)$.
Under the assumptions of Theorem~\ref{Y2attained}, we show that $\dim \big( span\big(u^{p_n-2}v_1, u^{p_n-2}v_2\big)\big)=2$.

\begin{proof}[End of the proof of Theorem \ref{Y2attained}]
If $\Lambda_G=+\infty$, then $H^2_1(M)$ is compactly  embedded in $L^{p_n}(M)$, therefore a minimizing sequence (up to a subsequence) converge to a generalized metric $u^{p_n-2}g$ that minimizes $Y^2_G(M,[g]_G)$. So let us assume that $\Lambda_G<+\infty$. Let $v_2^k$ be as above and $P$ be an arbitrary point of $M$. Its orbit has finite or infinite cardinality. In the following, we distinguish two cases, corresponding to  the finiteness of the orbit of $P$.
\medskip
\paragraph{\em Finite case} Let $O_G(P):=\{P_1,\dots ,P_K\}$ be the orbit of $P$ (not necessarily a minimal one) of cardinality $K\geq \Lambda_G$.  We define 
$$V_{k,i}=\eta_i|v_2^{k}|^{\varepsilon}v_2^k,$$
where $0<\varepsilon<(p_n-2)/2$, and $\eta_i$ is a cut-off function centered at $P_i$ and depends only on the distance to $P_i$, such that ${\rm supp}\ \eta_i\subset B_\delta(P_i) $ for all $i$, with $\delta$ smaller than the half distance between $P_i$ and $O_G(P_i)-\{P_i\}$. Hence the function 
\begin{equation}\label{AC}
V_{k}=\sum_{i=1}^KV_{k,i}
\end{equation}
is $G-$invariant, by construction.
We claim that we obtain the same inequality as (39) in \cite{Ammann-Humbert}. That is 
\begin{equation}\label{ineqAC} 
\|V_k\|_{p_n}^2\leq \big(\frac{1+\varepsilon^2}{1-3\varepsilon^2}\big)\frac{Y^2_G(M,[g]_G)}{Y(S^n)\Lambda_G^{\frac{2}{n}}}\Big(\int_{\cup_{i=1}^KB_{\delta}(P_i)} u_k^{p_n}dv_g\Big)^{\frac{2}{n}}\|V_k\|_{p_n}^2+C_{\delta},
\end{equation} 
where $\delta$ is small enough and $C_{\delta}$ is a constant that depends on $\delta$.

Indeed, doing similar computations as Step 1 in the proof of Theorem 1.4 of \cite{Ammann-Humbert}, we obtain:
\begin{equation}\label{eq 42}
\int_M\eta_i^2|v_2^{k}|^{2\varepsilon}v_2^kL_g(v_2^k)dv_g\geq \big(\frac{1-3\varepsilon^2}{1+\varepsilon^2}\big)
\int_M a_n| \nabla V_{k,i}|^2+B_0V_{k,i}^2dv_g-C_{\delta}.
\end{equation}
Note that \eqref{eq 42} corresponds to Equation~(42) in \cite{Ammann-Humbert}, which holds for $v_2^k$ and $V_{k,i}$ in our notation, since these two functions satisfy the same assumptions as the functions $w_m$ and $W_m$  used in \cite{Ammann-Humbert}. Therefore,
\begin{multline*}\sum_{i=1}^K\int_M\eta_i^2|v_2^{k}|^{2\varepsilon}v_2^kL_g(v_2^k)dv_g\geq  \big(\frac{1-3\varepsilon^2}{1+\varepsilon^2}\big)
\sum_{i=1}^K\int_M a_n| \nabla V_{k,i}|^2+B_0V_{k,i}^2dv_g\\
-KC_{\delta}.
\end{multline*}
By \eqref{v_2k}, we have that $$Y^2_G(M,[g]_G)\int_Mu^{p_n-2}V_k^{2}dv_g\geq \big(\frac{1-3\varepsilon^2}{1+\varepsilon^2}\big)
\int_M a_n| \nabla V_{k}|^2+B_0V_{k}^2dv_g-KC_{\delta}$$
and  by Theorem \ref{HVsob}, we have
\begin{equation}\label{desigI1}Y^2_G(M,[g]_G)\int_Mu^{p_n-2}_kV_k^{2}dv_g\geq \big(\frac{1-3\varepsilon^2}{1+\varepsilon^2}\big)Y(S^n)\Lambda_G^{\frac{2}{n}}\|V_k\|_{p_n}^2-KC_{\delta}.
\end{equation} 
Using the H\"{o}lder inequality in the left hand side of \eqref{desigI1}, we get \eqref{ineqAC}. 
\medskip
\paragraph{\em Infinite case} Let us assume that $P$ does not belong to a finite orbit. As in the proof of Theorem \ref{nodalsolution}, we  construct a function  $\zeta$, which is $G-$invariant and  with support in a tubular neighbourhood $T_{\delta}(O_G(P))$ of  $O_G(P)$, for $\delta>0$ sufficiently small. Let us consider the $G-$invariant function 
\begin{equation}\label{AC2}\tilde{V_k}=\zeta|v_2^{k}|^{\varepsilon}v_2^k.
\end{equation}  
Then,  using the same arguments as above, 
we obtain that 
\begin{equation}\label{ineqAC2} 
\|\tilde{V}_k\|_{p_n}^2\leq \big(\frac{1+\varepsilon^2}{1-3\varepsilon^2}\big)\frac{Y^2_G(M,g)}{Y(S^n)\Lambda_G^{\frac{2}{n}}}\Big(\int_{T_{\delta}(O_G(P))} u_k^{p_n}dv_g\Big)^{\frac{2}{n}}\|\tilde{V}_k\|_{p_n}^2+C_{\delta}.
\end{equation}

First, we assume that  $Y^2_G(M,[g]_G)<Y(S^n)\Lambda_G^{\frac{2}{n}}$.  Note that this assumption does not correspond to the one of Theorem \ref{Y2attained}. The general case is treated below.
For $\e$ small enough and by \eqref{ineqAC} and \eqref{ineqAC2}, we have that
\begin{gather} 
\|V_k\|_{p_n}^2\leq c\|V_k\|_{p_n}^2+C_{\delta}\label{ineqACi},\\
\|\tilde{V}_k\|_{p_n}^2\leq c\|\tilde{V}_k\|_{p_n}^2+C_{\delta}\label{ineqACii},
\end{gather} 
where $c<1$. These inequalities implies that $\{v_2^k\}$  is a bounded sequence in $L^{p_n+\varepsilon}(M)$. Since, $v_2^k$ converge strongly in $L^{p_n-\varepsilon}(M)$,  we can use the same arguments used in Step 2 (proof of Theorem 1.4, \cite{Ammann-Humbert}) to conclude that $v_2^k$ converge strongly to $v_2$ in $L^{p_n}(M)$. Hence, we obtain  that 
$$\int_Mu^{p_n-2}v_iv_jdv_g=\delta_{ij}.$$ 
This implies that $\dim\big(span \big(u^{p_n-2}v_1, u^{p_n-2}v_2\big)\big)=2$.

Now, we assume that 
$$Y^2_G(M,[g]_G)<\big(Y_G(M,[g]_G)^{\frac{n}{2}}+Y(S^n)^{\frac{n}{2}}\Lambda_G\big)^{\frac{2}{n}}.$$ As in Step 3
(proof of Theorem 1.4 in \cite{Ammann-Humbert}) we can use \eqref{ineqAC} and \eqref{ineqAC2} to prove that, there exists a closed set $\Omega\subset M$, such that for any open set $U$, with $\bar{U}\subseteq M-\Omega$, the  
sequences $\{v_1^k\}$ and $\{v_2^k\}$ converge strongly in $H^2_{1,G}(\bar{U})$ to $v_1$ and $v_2$, respectively.  $\Omega$ is the set of concentration points, {\em i.e.}, $P\in \Omega$, if for all $\delta>0$ $$\limsup_{k\to \infty }\int_{B_{\delta}(P)}u_k^{p_n}dv_g>\frac{1}{2}.$$ 
If $\Omega=\emptyset$, then by comments above implies that dimension of the subspace $span \big(u^{p_n-2}v_1, u^{p_n-2}v_2\big)$ is $2$, and the theorem follows.

Let assume that $\Omega\neq\emptyset$. Since $\|u_k\|_{p_n}=1$, the cardinality of $\Omega$ is at most one. 
 If $\Omega=\{P\}$, then $ O_G(P)=\{P\}$ and $\Lambda_G=1$, since the image of a concentration point by an isometry is also a concentration point. If $u^{p_n-2}v_1$ and $u^{p_n-2}v_2$ are not linearly independent, then
 we can mimic, without substantial changes due to the equivariant setting,  Step 4 (Actually Step 4, is true without assuming neither the linear dependence of $u^{p_n-2}v_1$ and $u^{p_n-2}v_2$ nor  $\Omega\neq \emptyset$), Step 5, and Step 6  of (proof of Theorem 1.4, \cite{Ammann-Humbert}), to obtain that 
$$Y^2_G(M,[g]_G)\geq \big(Y_G(M,[g]_G)^{\frac{n}{2}}+Y(S^n)^{\frac{n}{2}}\big)^{\frac{2}{n}}$$
which is a contradiction to our assumption. Therefore, $u^{p_n-2}v_1$ and $u^{p_n-2}v_2$ are  linearly independent, and  the theorem  follows in this case too.

\end{proof}

\end{document}